\documentclass[preprint,12pt]{elsarticle}
\usepackage{latexsym}
\usepackage{amssymb}
\usepackage{amsmath}
\usepackage{amsthm}
\usepackage{verbatim}
\usepackage{appendix}
\usepackage[text={145mm,220mm},centering]{geometry}
\usepackage{color}
\numberwithin{equation}{section}

\newtheorem{theorem}{Theorem}[section]
\newtheorem{lemma}{Lemma}[section]

\newtheorem{remark}{Remark}[section]

\newcommand{\R}{\mathbb{R}}

\newcommand{\Z}{\mathbb{Z}}
\newcommand{\C}{\mathbb{C}}
\newcommand{\N}{\mathcal{N}}
\newcommand{\V}{\mathcal{V}}
\journal{}
\begin{document}
\UseRawInputEncoding
\allowdisplaybreaks[4]
\begin{frontmatter}

%%%%%%%%%%%%%%%%%%%%%%%%%%%%%%%%%%ÌâÄ¿%%%%%%%%%%%%%%%%%%%%%%%%%%%%%%%%%
\title{Positive ground state solutions for generalized quasilinear Schr\"odinger equations with critical growth
\tnoteref{ack}}
\tnotetext[ack]{This work is partially supported by NSFC Grants (12225103, 12071065 and 11871140), the National Key Research and Development Program of China (2020YFA0713602 and 2020YFC1808301).}

%\tnoteref{ack}

%\tnotetext[ack]{This work is partially supported by NSFC Grants (nos. 11322105 and 11671071).}

%%%%%%%%%%%%%%%%%%%%%%%%%%%%%%×÷ÕßÐÅÏ¢%%%%%%%%%%%%%%%%%%%%%%%%%%%%%%%
\author[ad1]{Xin Meng}
\ead{mengxin22@mails.jlu.edu.cn}
\author[ad2,ad1]{Shuguan Ji\corref{cor}}
\ead{jishuguan@hotmail.com}
\address[ad1]{School of Mathematics, Jilin University, Changchun 130012, P.R. China}
\address[ad2]{School of Mathematics and Statistics and Center for Mathematics and Interdisciplinary Sciences, Northeast Normal University, Changchun 130024, P.R. China}
\cortext[cor]{Corresponding author.}

%%%%%%%%%%%%%%%%%%%%%%%%%%%%%%%%%%%%%%%%%%%%%%%%%%%%%%% ÕªÒª%%%%%%%%%%%%%%%%%%%%%%%%%%%%%%%%%%%%%%%
\begin{abstract}

This paper concerns the existence of positive ground state solutions for generalized quasilinear Schr\"odinger equations in $\R^{N}$ with critical growth which arise from plasma physics, as well as high-power ultrashort laser in matter. By applying a variable replacement, the quasilinear problem reduces to a semilinear problem which the associated functional is well defined in the Sobolev space $H^1(\R^N)$. We use the method of Nehari manifold for the modified equation, establish the minimax
characterization, then obtain each Palais-Smale sequence of the associated energy functional is bounded. By combining Lions's concentration-compactness lemma together with some classical arguments developed by Br\'ezis and Nirenberg \cite{bn}, we establish that the bounded Palais-Smale sequence has a nonvanishing behaviour. Finally, we obtain the existence of a positive ground state solution under some appropriate assumptions. Our results extend and generalize some known results.

\end{abstract}

\begin{keyword}
Ground state solutions, quasilinear Schr\"odinger equations, critical growth.
\end{keyword}

\end{frontmatter}

% \linenumbers

%% %%%%%%%%%%%%%%%%%%%%%%%%%%%%%%%%%%%%%%%%ÕýÎÄ%%%%%%%%%%%%%%%%%%%%%%%%%%%%%%%
\section{Introduction and main results}

In this paper, we are concerned with the existence of positive ground state solutions for the following generalized quasilinear Schr\"odinger equations with critical growth in the entire space
\begin{equation}\label{eq1}
-\mbox{div}\left(g^2(u)\nabla u\right)+g(u)g'(u)|\nabla u|^2+V(x)u=f(x,u)+g(u)|G(u)|^{2^{*}-2}G(u)
\end{equation}
for $x=(x_1, x_2, \cdots, x_N)\in\R^{N}$, where $N\geq3$, $2^{*}=\frac{2N}{N-2}$ is the Sobolev critical exponent, $G(t)=\int_{0}^{t}g(\tau)d\tau$, $V(x)\in C(\R^{N},\R)$ and $f(x,u)\in C(\R^{N}\times\R,\R)$ are periodic in $x$ with same period for each component $x_i$, and $g$ satisfies that\\
($g$) $g\in C^1(\R,\R^+)$ is an even function with $g'(t)\geq0$ for all $t\geq0$ and $g(0)=1$.

As far as we know, the problem is related to the existence of solitary wave solutions for quasilinear Schr\"odinger equations of the form

\begin{equation}\label{eq1.2}
i\partial_tz=-\triangle z+W(x)z-k(x,z)-\triangle l(|z|^2)l'(|z|^2)z,\quad x\in\R^N,
\end{equation}
where $z: \R\times\R^N\rightarrow\C$, $W:\R^N\rightarrow\R$ is a given potential, $k$ and $l$ are suitable real functions. Quasilinear equations of form \eqref{eq1.2} appear more naturally in mathematical physics and have been derived as models of several physical phenomena for various types of $l$. In particular, the case $l(s)=s$ models the time evolution of the condensate wave function in superfluid film (see \cite{SK} and \cite{EWL}), and such a model is called the superfluid film equation in plasma physics by Kurihara in \cite{SK}. For the case $l(s)=(1+s)^{1/2}$, equation \eqref{eq1.2} models the self-channeling of a high-power ultrashort laser in matter, the propagation of a high-irradiance laser in a plasma creates an optical index depending nonlinearly on the light intensity, and this leads to interesting new nonlinear wave equations (see \cite{HSB,XLC,ADB,BR}). We see that \eqref{eq1.2} also appears in the theory of Heisenberg ferromagnetism and magnons \cite{AMK}, in dissipative quantum mechanics \cite{RWH} and in condensed matter theory \cite{VGM}. We can refer the readers to see \cite{S1} and \cite{S2}, and references therein for more details.

By seeking solutions of the type of standing waves, namely, the solutions of the form $z(t,x)=\exp(-iBt)u(x)$ with a constant $B\in\R$ and a real function $u$, \eqref{eq1.2} can be reduced to the following equation
\begin{equation}\label{eq1.3}
-\triangle u+V(x)u-\triangle l(u^2)l'(u^2)u=k(x,u),\quad x\in\R^N,
\end{equation}
where $V(x)=W(x)-B$. Notice that if we take
\begin{equation*}
g^2(u)=1+\frac{[(l(u^2))']^2}{2},
\end{equation*}
then \eqref{eq1.3} turns into \eqref{eq1} with $k(x,u)=f(x,u)+g(u)|G(u)|^{2^{*}-2}G(u)$.

Moreover, if we assume that $g(u)$ is a positive constant, then \eqref{eq1} arises in biological models and propagation of laser beams \cite{PLK}. If we set $g^2(u)=1+2u^2$, i.e., $l(s)=s$, we get the superfluid film equation in plasma physics with the form
\begin{equation}\label{eq1.4}
-\triangle u+V(x)u-\triangle(u^2)u=k(x,u),\quad x\in\R^N.
\end{equation}
Recent mathematical studies have extensively focused on the existence of solutions for \eqref{eq1.4} in subcritical and critical cases. For subcritical cases, Liu and Wang in \cite{S1} establish the existence of ground states of soliton type solutions through a constrained minimization argument, the method is also used in \cite{CXJJGA, TANGJGA} to obtain ground state solutions. In \cite{S2}, they use a change of variables for the first time to reformulate the quasilinear problem, then obtain a semilinear problem which has an associated functional well-defined in the Sobolev space $H^1(\R^N)$. Meanwhile, an Orlicz space framework is used to prove the existence of positive solutions for equation \eqref{eq1.4} with a parameter family of special nonlinear terms via mountain pass theorem (see \cite{AR}). The same method of changing variables in \cite{S2} is also used in \cite{FS}, they obtain the existence of infinitely many geometrically distinct solutions. For the critical case, Moameni in \cite{AM} considers the related perturbed problem with strongly singular nonlinearities and obtains a positive radial solution when the potential $V$ is radial and satisfies some geometric conditions. In \cite{LLW}, Liu et al. propose a new perturbation approach to treat the critical exponent case and give new existence results.

If we set $g^2(u)=1+\frac{u^2}{2(1+u^2)}$, i.e., $l(s)=(1+s)^{1/2}$, then \eqref{eq1} becomes the following equation
\begin{equation}\label{eq1.5}
-\triangle u+V(x)u-[\triangle(1+u^2)^{\frac{1}{2}}]\frac{u}{2(1+u^2)^{\frac{1}{2}}}=k(x,u),\quad x\in\R^N,
\end{equation}
which is used as a model of the self-channeling of a high-power ultrashort laser in matter. In \cite{ADB}, the authors prove the local existence in one dimensional transverse space, the global existence and uniqueness of small solutions in two and three dimensional transverse space for equation \eqref{eq1.5}. It is worth pointing out that \eqref{eq1} is more general than \eqref{eq1.4} and \eqref{eq1.5}.

Note that the results mentioned above on the existence of solitary wave solutions for Schr\"odinger equations \eqref{eq1.2} are for certain given special function $l(s)$. For some general functions $l(s)$, Shen and Wang in \cite{Shen} introduce a new variable replacement method which provides us a unified method in studying problem \eqref{eq1.2}. They establish the existence of positive solutions for \eqref{eq1.2} with subcritical nonlinear terms. Later, the results are extended by \cite{Furtado} for the existence and multiplicity of solutions, and by \cite{CCPAA} for the existence of ground state solutions. As for generalized quasilinear Schr\"odinger equations with critical growth, Deng et al. in \cite{Peng} establish the existence of positive solutions by using the mountain pass theorems. To the best of our knowledge, no one has studied the existence of ground state solutions for generalized quasilinear Schr\"odinger equations with critical growth. The main difficulty is owing to the possible lack of compactness for the criticality of the growth and the unboundedness of the definition domain.

Inspired by the previous work, the goal of this paper is to establish the existence of positive ground state solutions of the generalized quasilinear equations \eqref{eq1} with critical growth. We develop a more direct and simpler approach to obtain the ground state solutions which have great physical interests. First, we use a change of variable to reduce the problem \eqref{eq1} to a semilinear problem which the associated energy functional is well defined in $H^1(\R^N)$. Then we use the method of Nehari manifold developed by Szulkin and Weth in \cite{AS,monm} which gives rise to a new minimax characterization of the corresponding critical value to conquer it and obtain each Palais-Smale sequence of the associated energy functional is bounded. By combining Lions's concentration-compactness lemma together with some classical arguments developed by Br\'ezis and Nirenberg \cite{bn}, we establish that the bounded Palais-Smale sequence has a nonvanishing behaviour. Using this fact, we obtain a translated Palais-Smale sequence converges to a nontrivial critical point of the associated functional. Finally, we obtain the existence of a positive ground state solution of equation \eqref{eq1}.

Compared with the existing results, our results improve and generalize the existing related results. In the present paper, we concern the unified periodicity of the potential function $V$ and the nonlinear term $f$, and without lose of generality we assume that they are both 1-periodic in $x$. Besides, the nonlinear term $f$ in our paper need not be differentiable as mentioned in \cite{Peng}, then the constrained manifold need not be of class $C^1$ in our case. Furthermore, our assumptions on the nonlinear term $f$ are somewhat weaker than those in \cite{Peng} and \cite{Shen}, and we do not require the classical Ambrosetti-Rabinowitz superlinear condition.

Before introducing the main theorem, we will give some basic assumptions. We assume that the functions $g, V$ and $f$ satisfy the following assumptions:
\\($V$) $V(x)\in C(\R^N,\R)$ is 1-periodic in $x_i$, $1\leq i\leq N$, and $V(x)>0$ for all $x\in\R^N$;
\\($f_1$) $f(x,t)\in C(\R^N\times\R,\R)$ is 1-periodic in $x_i$, $1\leq i\leq N$;
\\($f_2$) $\lim\limits_{t\rightarrow+\infty}\frac{f(x,t)}{g(t)|G(t)|^{2^{*}-1}}=0$ and $\lim\limits_{t\rightarrow0}\frac{f(x,t)}{g(t)G(t)}=0$ uniformly in $x\in\R^N$;
\\($f_3$) $t\rightarrow\frac{f(x,t)}{g(t)G(t)}$ is non-decreasing on $(0,+\infty)$;
\\($f_4$) there exists an open bounded set $\Omega\subset\R^N$, such that

$\lim\limits_{t\rightarrow+\infty}\frac{F(x,t)}{G(t)^4}=+\infty$ uniformly for $x\in\Omega$, if $N=3$,

$\lim\limits_{t\rightarrow+\infty}\frac{F(x,t)}{G(t)^2\ln G(t)}=+\infty$ uniformly for $x\in\Omega$, if $N=4$,

$\lim\limits_{t\rightarrow+\infty}\frac{F(x,t)}{G(t)^2}=+\infty$ uniformly in $x\in\R^N$, if $N\geq5$, where $F(x,t)=\int_{0}^{t}f(x,\tau)d\tau$.

We note that if $u_0$ is a solution of \eqref{eq1}, then so is the element $u_0(\cdot-k)$ under the action of $\Z^N$, set $\mathcal{O}(u_0)=\{u_0(\cdot-k):k\in\Z^N\}$, $\mathcal{O}(u_0)$ is called the orbit of $u_0$ with respect to the action of $\Z^N$. Two solutions $u_1$ and $u_2$ of \eqref{eq1} are said to be geometrically distinct if $\mathcal{O}(u_1)$ and $\mathcal{O}(u_2)$ and disjoint. Now we state our main result.

\begin{theorem}\label{th1}
Suppose that $(g),(V)$ and $(f_1)-(f_4)$ hold. Then problem \eqref{eq1} has a positive ground state solution.
\end{theorem}

This paper is organized as follows. In Section 2, we first introduce some basic spaces and use a change of variable to reduce the quasilinear problem to a semilinear problem which has an associated functional well defined in $H^1(\R^N)$, then we present some useful preliminary lemmas. In Section 3, we give estimates on the minimax level value of the functional associated with the modified equation, which ensure that this level value is below of a certain critical values. In Section 4, we give the proof of our main Theorem \ref{th1}.

\section{Preliminaries}

In this section, we will introduce some basic spaces. By applying a change of variable, the quasilinear problem is reduced to a semilinear problem, which the associated functional is well defined in $H^1(\R^N)$. Based on our basic assumptions, we will give some properties for functions $G, G^{-1}$ and $h$ which will be defined in \eqref{eq7}. Then we present some useful lemmas for the method of Nehari manifold.

Now let us recall some basic notions. We denote by $L^p(\R^N)$ the usual Lebesgue space endowed with the norms $\left\|u\right\|_p=(\int_{\R^N}|u|^pdx)^{\frac{1}{p}}$, $1\leq p<\infty$ and $\left\|u\right\|_\infty=\mbox{ess}\sup\limits_{x\in\R^N}|u(x)|$. Let $H^1(\R^N)$ denotes the Sobolev space modeled in $L^2(\R^N)$ with its usual norm $\left\|u\right\|_{H^1}=\left(\int_{\R^N}(|\nabla u|^2+u^2)dx\right)^{\frac{1}{2}}$.

It is well known that in the study of the elliptic equations, the potential function V plays an important role in dealing with the compactness problem. Let us introduce the following working space
\begin{equation*}
E=\{u\in H^1(\R^N): \int_{\R^N}V(x)u^2dx<\infty\},
\end{equation*}
endowed with the norm
\begin{equation*}
\left\|u\right\|=\left(\int_{\R^N}(|\nabla u|^2+V(x)u^2)dx\right)^{\frac{1}{2}}.
\end{equation*}
It has been proved in \cite{Tang} that if the assumption $(V)$ holds, then there exist two positive constants $d_1$ and $d_2$ such that $d_1\|u\|_{H^1}^2\leq\|u\|^2\leq d_2\|u\|_{H^1}^2$ for all $u\in E$. Therefore from the Sobolev embedding, we obtain that $E\hookrightarrow L^\gamma(\R^N)$ is continuous for any $\gamma\in[2,2^*]$.

We observe that problem \eqref{eq1} is formally the Euler-Lagrange equation associated with the energy functional

\begin{eqnarray*}
J(u)&=&\frac{1}{2}\int_{\R^N}g^2(u)|\nabla u|^2dx+\frac{1}{2}\int_{\R^N}V(x)|u|^2dx -\int_{\R^N}F(x,u)dx
\\&&-\frac{1}{2^*}\int_{\R^N}|G(u)|^{2^*}dx.
\end{eqnarray*}

From the variational point of view, $J$ may not be well defined in $H^1(\R^N)$, hence our first difficulty associated with problem \eqref{eq1} is finding an appropriate function space where the functional $J$ is well defined. To overcome this difficulty, we use a unified new change of variable constructed by Shen and Wang in \cite{Shen}, as

\begin{equation*}
v=G(u)=\int_{0}^{u}g(t)dt.
\end{equation*}
Then we obtain the following functional
\begin{eqnarray}
\nonumber I(v)&=&\frac{1}{2}\int_{\R^N}|\nabla v|^2dx+\frac{1}{2}\int_{\R^N}V(x)|G^{-1}(v)|^2dx- \int_{\R^N}F(x,G^{-1}(v))dx
\\&&-\frac{1}{2^*}\int_{\R^N}|v|^{2^*}dx,\label{eq2}
\end{eqnarray}
where $G^{-1}(v)$ is the inverse function of $G(u)$. Since $g$ is a nondecreasing positive function, we get $|G^{-1}(v)|\leq\frac{1}{g(0)}|v|$ which will be proved in detail later. From this and our assumptions, it is clear that $I$ is well defined in $H^1(\R^N)$ and $I\in C^1$.

If u is a nontrivial solution of \eqref{eq1}, then it should satisfy
\begin{eqnarray}
\nonumber\langle J'(u), \psi\rangle&=&\int_{\R^N}\left[g^2(u)\nabla u\nabla \psi+g(u)g'(u)|\nabla u|^2\psi+V(x)u\psi-f(x,u)\psi\right]dx
\\\nonumber&&+\int_{\R^N}g(u)|G(u)|^{2^*-2}G(u)\psi dx
\\&=&0 \label{eq3}
\end{eqnarray}
for all $\psi\in C_0^\infty(\R^N)$.

We show that for all $\varphi\in C_0^\infty(\R^N)$, \eqref{eq3} is equivalent to
\begin{eqnarray}
\nonumber\langle I'(v), \varphi\rangle&=&\int_{\R^N}\left[\nabla v\nabla\varphi+V(x)\frac{G^{-1}(v)}{g(G^{-1}(v))}\varphi- \frac{f(x,G^{-1}(v))}{g(G^{-1}(v))}\varphi\right]dx
\\\nonumber&&-\int_{\R^N}|v|^{2^*-2}v\varphi dx
\\&=&0. \label{eq4}
\end{eqnarray}
Indeed, if we choose $\psi=\frac{1}{g(u)}\varphi$ in \eqref{eq3}, then we immediately get \eqref{eq4}. On the other hand, since $u=G^{-1}(v)$, if we let $\varphi=g(u)\psi$ in \eqref{eq4} then we get \eqref{eq3}. Therefore, in order to find the nontrivial solutions of \eqref{eq1}, it is sufficient to study the existence of the nontrivial solutions of the following modified equations
\begin{equation}\label{eq5}
-\triangle v+V(x)\frac{G^{-1}(v)}{g(G^{-1}(v))}-\frac{f(x,G^{-1}(v))}{g(G^{-1}(v))}-|v|^{2^*-2}v=0.
\end{equation}

We can easily verify that \eqref{eq1} is equivalent to problem \eqref{eq5} and the nontrivial critical points of $I(v)$ are the nontrivial solutions of problem \eqref{eq5}. Hence in the following, we only need to find ground state solutions for problem \eqref{eq5}.

Since we are going to discuss the existence of positive ground state solutions of problem \eqref{eq5}, we  assume that $f(x,t)=0$ for all $(x,t)\in\R^N\times\R^-$. We claim that all nontrivial critical points of $I$ are the positive solutions of \eqref{eq5}. In fact, let $v\in H^1(\R^N)$ be a nontrivial critical point of $I$, if we take $\varphi=-v^-$ in \eqref{eq4}, where $v^-:=\max\{-v,0\}$, we obtain
\begin{equation*}
\int_{\R^N}\left(|\nabla v^-|^2-|v^-|^{2^*}+V(x)\frac{G^{-1}(v)}{g(G^{-1}(v))}(-v^-)\right)dx=0.
\end{equation*}
Since $G^{-1}(v)(-v^-)\geq0$, we have
\begin{equation*}
\int_{\R^N}(|\nabla v^-|^2-|v^-|^{2^*})dx=0\quad \mbox{and}\quad \int_{\R^N}V(x)\frac{G^{-1}(v)}{g(G^{-1}(v))}(-v^-)dx=0.
\end{equation*}
Hence we conclude that $v^-=0$ almost everywhere in $\R^N$, therefore we have $v=v^+\geq0$ where $v^+:=\max\{v,0\}$. Then $v$ must be a nonnegative solution of
\begin{equation}\label{eq6}
-\triangle v+V(x)v=V(x)\left(v-\frac{G^{-1}(v)}{g(G^{-1}(v))}\right)+\frac{f(x,G^{-1}(v))}{g(G^{-1}(v))}+ (v^+)^{2^*-1}.
\end{equation}
Standard regularity argument shows that $v\in C^2(\R^N)$. Moreover, the assumption $(g), (f_2)$ and Lemma \ref{le1} $(2)$ show that the right hand side of \eqref{eq6} is nonnegative since
\begin{equation*}
v-\frac{G^{-1}(v)}{g(G^{-1}(v))}=G(u)-\frac{u}{g(u)}=\frac{G(u)g(u)-u}{g(u)}
\end{equation*}
and
\begin{equation*}
G(u)g(u)>\frac{G^2(u)}{u}>u.
\end{equation*}
Therefore, by the strong maximum principle in \cite{smp}, we know that $v$ is positive.

We define the Nehari manifold
\begin{equation*}
\N:=\{v\in E\setminus\{0\}: \langle I'(v), v\rangle=0\}.
\end{equation*}
Since we do not know whether $\N$ is of class $C^1$ under our assumptions, we cannot use the minimax theory directly on $\N$. To overcome this difficulty, we employ a similar argument developed by Szulkin and Weth in \cite{AS,monm}.

Here we present some useful lemmas before we prove the existence of positive ground state solutions for problem \eqref{eq1}.
\begin{lemma}\label{le1}
We note that $s=G(t)$. The functions $g(t), G(t)=\int_0^tg(\tau)d\tau$ and $G^{-1}(s)$ satisfy the following properties by the assumption $(g)$:

$(1)$ the functions $G(t)$ and $G^{-1}(s)$ are strictly increasing and odd;

$(2)$ $G(t)\leq g(t)t$ for all $t\geq0$;

$(3)$ $|G^{-1}(s)|\leq\frac{1}{g(0)}|s|=|s|$ for all $s\in\R$;

$(4)$ $\frac{G^{-1}(s)}{sg(G^{-1}(s))}$ is non-increasing for $s>0$;

$(5)$ $\frac{G^{-1}(s)}{s}$ is non-increasing on $(0,+\infty)$ and
\begin{equation*}
\lim_{s\rightarrow0}\frac{G^{-1}(s)}{s}=\frac{1}{g(0)},\quad  \lim_{s\rightarrow+\infty}\frac{G^{-1}(s)}{s}=\left\{\begin{matrix}\frac{1}{g(\infty)}&if\ g\ is\ bounded; \\0&\ \ \ if\  g\ is\ unbounded.\end{matrix}\right.
\end{equation*}

\end{lemma}
\begin{proof}
The points $(1)$ and $(2)$ are immediate by the definitions of $G(t)$ and the assumption ($g$).

By the differential mean value theorem, we have
\begin{equation*}
|G^{-1}(s)|=|G^{-1}(s)-G^{-1}(0)|=(G^{-1})^{'}(\xi)|s|=\frac{1}{g(G^{-1}(\xi))}|s|\leq\frac{1}{g(0)}|s| =|s|
\end{equation*}
for all $s\in\R$, where $\xi\in(0,s)$, then $(3)$ is proved.

Since $G(t)\leq g(t)t$ and $g'(t)\geq0$ for all $t\geq0$, we have that
\begin{equation*}
G(t)\left(\frac{1}{g^2(t)}\left(g(t)-g'(t)t\right)\right)\leq g(t)t\left(\frac{1}{g^2(t)}\left(g(t)-g'(t)t\right)\right)\leq t\quad \mbox{for}\ t\geq0,
\end{equation*}
which gives that
\begin{equation*}
G(t)\left(\frac{t}{g(t)}\right)_t^{'}\frac{1}{g(t)}\leq\frac{t}{g(t)}.
\end{equation*}
Note that $s=G(t)$, for $s>0$ we have
\begin{equation*}
G(t)\left(\frac{t}{g(t)}\right)_s^{'}\leq\frac{t}{g(t)},
\end{equation*}
and hence
\begin{equation*}
s\left(\frac{G^{-1}(s)}{g(G^{-1}(s))}\right)_s^{'}\leq\frac{G^{-1}(s)}{g(G^{-1}(s))}.
\end{equation*}
Denote $q(s)=\frac{G^{-1}(s)}{g(G^{-1}(s))}$, therefore we have $sq'(s)\leq q(s)$ for $s>0$. Let $Q(s)=\frac{q(s)}{s}$, then by a simple process of differentiation, we obtain that
\begin{equation*}
Q'(s)=\frac{q'(s)s-q(s)}{s^2}\leq0\quad\mbox{for}\ s>0.
\end{equation*}
Then we have point $(4)$.

It follows from point $(2)$ that
\begin{equation*}
\left(\frac{(G^{-1}(s)}{s}\right)_s^{'}=\frac{1}{g(t)}\left(\frac{t}{G(t)}\right)_t^{'}=
\frac{1}{g(t)G^2(t)}(G(t)-g(t)t)\leq0\quad\mbox{for all}\ t\geq0,
\end{equation*}
and using the L'Hospital rule we know point $(5)$ is satisfied. The proof is completed.
\end{proof}

Denote
\begin{equation}\label{eq7}
h(x,s)=V(x)s-V(x)\frac{G^{-1}(s)}{g(G^{-1}(s))}+\frac{f(x,G^{-1}(s))}{g(G^{-1}(s))},
\end{equation}
then
\begin{equation}\label{eq8}
H(x,s)=\int_0^sh(x,\tau)d\tau=\frac{1}{2}V(x)[s^2-(G^{-1}(s))^2]+F(s,G^{-1}(s)).
\end{equation}
Consequently, from \eqref{eq2} we have
\begin{eqnarray}
\nonumber I(v)&=&\frac{1}{2}\int_{\R^N}\left(|\nabla v|^2+V(x)v^2\right)dx-\int_{\R^N}H(x,v)dx-\frac{1}{2^{*}}\int_{\R^N}|v|^{2^{*}}dx
\\&=&\frac{1}{2}\|v\|^2-\int_{\R^N}H(x,v)dx-\frac{1}{2^{*}}\int_{\R^N}|v|^{2^{*}}dx. \label{eq9}
\end{eqnarray}
\begin{lemma}\label{le2}
Assume that $(g), (V)$ and $(f_1)-(f_4)$ hold, the functions $h(x,s)$ and $H(x,s)$ enjoy the following properties:

$(1)$ $\lim\limits_{s\rightarrow0^+}\frac{h(x,s)}{s}=0$ and $\lim\limits_{s\rightarrow+\infty}\frac{h(x,s)}{s^{2^*-1}}=0$ uniformly in $x\in\R^N$;

$(2)$ $\lim\limits_{s\rightarrow0^+}\frac{H(x,s)}{s^2}=0$ and $\lim\limits_{s\rightarrow+\infty}\frac{H(x,s)}{s^{2^*}}=0$ uniformly in $x\in\R^N$;

$(3)$ $s\mapsto\frac{h(x,s)}{s}$ is non-decreasing on $(0,+\infty)$;

$(4)$ for an open bounded set $\Omega\subset\R^N$,

\ \ \ \ $\lim\limits_{s\rightarrow+\infty}\frac{H(x,s)}{s^4}=+\infty$ uniformly for $x\in\Omega$, if $N=3$,

\ \ \ \ $\lim\limits_{s\rightarrow+\infty}\frac{H(x,s)}{s^2\ln s}=+\infty$ uniformly for $x\in\Omega$, if $N=4$,

\ \ \ \ $\lim\limits_{s\rightarrow+\infty}\frac{H(x,s)}{s^2}=+\infty$ uniformly for $x\in\R^N$, if $N\geq5$;

$(5)$ $\frac{1}{2}h(x,s)s\geq H(x,s)\geq0$ for all $(x,s)\in\R^N\times\R$;

$(6)$For any $\delta>0$, there exists $C_\delta>0$ and $\alpha\in(2,2^*)$ such that

\ \ \ \ $0\leq h(x,s)\leq \delta|s|+C_\delta|s|^{2^*-1}$, for any $(x,s)\in\R^N\times\R$,

\ \ \ \ $0\leq H(x,s)\leq \frac{\delta}{2}|s|^2+\frac{C_\delta}{2^*}|s|^{2^*}$, for any $(x,s)\in\R^N\times\R$,

\ \ \ \ $0\leq H(x,s)\leq \frac{\delta}{2}|s|^2+\frac{\delta}{2^*}|s|^{2^*}+\frac{C_\delta}{\alpha}|s|^\alpha$, for any $(x,s)\in\R^N\times\R$.

\end{lemma}
\begin{proof}
Note that $s=G(t)$, from $(g), (f_2)$ and Lemma \ref{le1} $(5)$, we have
\begin{eqnarray*}
\lim_{s\rightarrow0^+}\frac{h(x,s)}{s}&=& \lim_{s\rightarrow0^+}\left[V(x)\left(1-\frac{G^{-1}(s)}{sg(G^{-1}(s))}\right) +\frac{f(x,G^{-1}(s))}{sg(G^{-1}(s))}\right]
\\&=&V(x)\left(1-\frac{1}{g^2(0)}\right)+\lim_{t\rightarrow0^+}\frac{f(x,t)}{g(t)G(t)}
\\&=&0,
\end{eqnarray*}
and
\begin{eqnarray*}
\lim_{s\rightarrow+\infty}\frac{h(x,s)}{s^{2^*-1}}&=&\lim_{s\rightarrow+\infty}\left[\frac{V(x)} {s^{2^*-2}} -V(x)\frac{G^{-1}(s)}{sg(G^{-1}(s))}\frac{1}{s^{2^*-2}}+\frac{f(x,G^{-1}(s))}{g(G^{-1}(s))}\frac{1} {s^{2^*-1}}\right]
\\&=&-V(x)\lim_{s\rightarrow+\infty}\frac{G^{-1}(s)}{sg(G^{-1}(s))}\frac{1}{s^{2^*-2}}+
\lim_{t\rightarrow+\infty}\frac{f(x,t)}{g(t)G(t)^{2^*-1}}
\\&=&0.
\end{eqnarray*}
Therefore, point $(1)$ is proved and point $(2)$ is obvious. We can easily verify that point $(3)$ is according to the assumption $(f_3)$ and Lemma \ref{le1} $(4)$, and point $(4)$ follows from the assumption $(f_4)$.

From $(f_3)$ and Lemma \ref{le1} $(2)$, we have
\begin{eqnarray*}
H(x,s)&=&\int_0^sh(x,\tau)d\tau=\int_0^s\frac{h(x,\tau)}{g(\tau)G(\tau)}g(\tau)G(\tau)d\tau
\\&\leq&\frac{h(x,s)}{g(s)G(s)}\int_0^sG(\tau)dG(\tau)=\frac{h(x,s)G(s)}{2g(s)}\leq\frac{h(x,s)s}{2},
\end{eqnarray*}
then point $(5)$ is proved. Points $(1)$ and $(2)$ imply point $(6)$ directly. Hence the proof is completed.
\end{proof}
\begin{lemma}\label{le3}
For each $v\in E\setminus\{0\}$, there exist a unique $t_v=t(v)>0$ such that $t_vv\in\N$ and $I(t_vv)=\max\limits_{t>0}I(tv)$.

\end{lemma}
\begin{proof}
For $t>0$, we define
\begin{equation*}
M(t)=I(tv)=\frac{t^2}{2}\|v\|^2-\int_{\R^N}H(x,tv)dx-\frac{t^{2^{*}}}{2^{*}}\int_{\R^N}|v|^{2^{*}}dx.
\end{equation*}

It follows from Lemma \ref{le2} $(6)$ and the Sobolev inequality that
\begin{equation}\label{eq10}
\int_{\R^N}H(x,tv)dx\leq\frac{\delta}{2}t^2\int_{\R^N}|v|^2dx+ \frac{C_\delta}{2^*}t^{2^*}\int_{\R^N}|v|^{2^*}dx\leq C_1\delta t^2\|v\|^2+C_\delta t^{2^*}\|v\|^{2^*},
\end{equation}
where the positive constant $C_1$ is independent of $t$. Thus for $\delta>0$ sufficiently small, we have
\begin{eqnarray*}
M(t)&\geq&\frac{t^2}{2}\|v\|^2-C_1\delta t^2\|v\|^2-(C_\delta+ \frac{1}{2^*})t^{2^*}\|v\|^{2^*}
\\&\geq&\frac{t^2}{2}\|v\|^2-C_2t^{2^*}\|v\|^{2^*},
\end{eqnarray*}
where the positive constant $C_2$ is independent of $t$. Since $v\neq0$ and $2^*=\frac{2N}{N-2}>2$ for $N\geq3$, it is easy to see that $M(t)>0$ whenever $t>0$ is small enough. On the other hand, by using Lemma \ref{le2} $(5)$, we have
\begin{equation*}
M(t)\leq\frac{t^2}{2}\|v\|^2-\frac{t^{2^*}}{2^*}\int_{\R^N}|v|^{2^*}dx,
\end{equation*}
then we can easily show that $M(t)\rightarrow-\infty$ as $t\rightarrow+\infty$. So there exist $t_v=t(v)>0$ such that $M(t_v)=\max\limits_{t>0}M(t)$ and $M'(t_v)=0$, i.e., $I(t_vv)=\max\limits_{t>0}I(tv)$ and then $t_vv\in\N$.

Suppose that there exist $t_1>t_2>0$ such that $t_1v\in\N$ and $t_2v\in\N$. Then by Lemma \ref{le2} $(3)$,  one has
\begin{eqnarray*}
\|v\|^2&=&\int_{\R^N}\frac{h(x,t_1v^+)}{t_1v^+}(v^+)^2dx+t_1^{2^*-2}\int_{\R^N}|v|^{2^*}dx
\\&>&\int_{\R^N}\frac{h(x,t_2v^+)}{t_2v^+}(v^+)^2dx+t_2^{2^*-2}\int_{\R^N}|v|^{2^*}dx
\\&=&\|v\|^2,
\end{eqnarray*}
which is a contradiction. Hence $t_v$ is unique and this completes the proof.
\end{proof}
\begin{lemma}\label{le4}
$(1)$ There exist $\rho>0$ such that level value $c=\inf\limits_\N I\geq\inf\limits_{S_\rho}I>0$, where $S_\rho=\{v\in E: \|v\|=\rho\}$.

$(2)$ $\|v\|^2\geq2c$ for all $v\in \N$.
\end{lemma}
\begin{proof}
$(1)$ By \eqref{eq10} we have
\begin{equation*}
\int_{\R^N}H(x,v)dx\leq\frac{\delta}{2}\int_{\R^N}|v|^2dx+\frac{C_\delta}{2^*}\int_{\R^N}|v|^{2^*}dx\leq C_1\delta\|v\|^2+C_\delta\|v\|^{2^*},
\end{equation*}
therefore from \eqref{eq9} we choose $\delta>0$ sufficiently small, then
\begin{eqnarray*}
I(v)&\geq&\frac{1}{2}\|v\|^2-C_1\delta\|v\|^2-C_\delta\|v\|^{2^*}-\frac{1}{2^*}\|v\|^{2^*}
\\&\geq&\frac{1}{2}\|v\|^2-C_2\|v\|^{2^*},
\end{eqnarray*}
where positive constant $C_1$ and $C_2$ are as in Lemma \ref{le3}. Hence $\inf\limits_{S_\rho}I>0$ is obtained when $\rho$ is small enough. The inequality $\inf\limits_\N I\geq\inf\limits_{S_\rho}I$ is a consequence of Lemma \ref{le3} since for every $v\in\N$ there is $e>0$ such that $ev\in S_\rho$ and $I(t_vv)\geq I(ev)$.

$(2)$ For $v\in\N$, by \eqref{eq9} we have
\begin{eqnarray*}
c&\leq&I(v)
\\&=&\frac{1}{2}\int_{\R^N}\left(|\nabla v|^2+V(x)v^2\right)dx-\int_{\R^N}H(x,v)dx-\frac{1}{2^{*}}\int_{\R^N}|v|^{2^{*}}dx
\\&\leq&\frac{1}{2}\int_{\R^N}\left(|\nabla v|^2+V(x)v^2\right)dx
\\&=&\frac{1}{2}\|v\|^2.
\end{eqnarray*}
Hence $\|v\|^2\geq2c$ for all $v\in \N$, then the proof is completed.
\end{proof}

\begin{lemma}\label{le5}
If $\V$ is a compact subset of $E\setminus\{0\}$, then $m$ maps $\V$ into bounded set in $E$.
\end{lemma}
\begin{proof}
We may assume without loss of generality that $\V\subset S_1$, where $S_1$ is the unit sphere in $E$. Arguing by contradiction, suppose that there exist $u_n\in\V$ and $v_n=t_{u_n}u_n$ such that $\|v_n\|\rightarrow\infty$ as $n\rightarrow\infty$. Passing to a subsequence, there is $u\in E$ with $\|u\|=1$ such that $u_n\rightarrow u\in S_1$. Since $|v_n(x)|\rightarrow\infty$ if $u(x)\neq0$, then by Lemma \ref{le2} $(4)$ and Fatou's lemma that
\begin{equation*}
\int_{\R^N}\frac{H(x,v_n)}{t_{u_n}^2}=\int_{\R^N}\frac{H(x,v_n)}{v_n^2}u_n^2\rightarrow\infty.
\end{equation*}
By \eqref{eq9}, we have
\begin{equation*}
0\leq\frac{I(v_n)}{t_{u_n}^2}\leq\frac{1}{2}-\int_{\R^N}\frac{H(x,v_n)}{t_{u_n}^2}\rightarrow-\infty,
\end{equation*}
which is a contradiction. Thus the Lemma is proved.
\end{proof}

Recall that $S_1$ is the unit sphere in $E$ and we define the mapping $m: S_1\rightarrow\N$ by setting
\begin{equation*}
m(w):=t_ww,
\end{equation*}
where $t_w$ is as in Lemma \ref{le3}. Noting that $\|m(w)\|=t_w$. Lemma \ref{le6} and Lemma \ref{le7} below are taken directly from \cite{monm} (see Proposition 8 and Corollary 10 there). The hypotheses in \cite{monm} are satisfied by Lemma \ref{le3}, \ref{le4} and \ref{le5} above. Indeed, if $M(t)=I(tw)$ and $w\in S_1$, then $M'(t)>0$ for $0<t<t_w$ and $M'(t)<0$ for $t>t_w$ by Lemma \ref{le3}, $t_w\geq\zeta>0$ by Lemma \ref{le4} and $t_w\leq R$ for $w\in\V\subset S_1$ by Lemma \ref{le5}, where $\zeta$ and $R$ are positive constant.

\begin{lemma}\label{le6}
The mapping $m$ is continuous. Moreover, the mapping $m$ is a homeomorphism between $S_1$ and $\N$, and the inverse of $m$ is given by $m^{-1}(v)=\frac{v}{\|v\|}$.
\end{lemma}

We shall consider the functional $\Psi: S_1\rightarrow\R$ given by
\begin{equation*}
\Psi(w):=I(m(w)).
\end{equation*}

\begin{lemma}\label{le7}
$(1)$ $\Psi\in C^1(S_1,\R)$ and
\begin{equation*}
\langle\Psi'(w), z\rangle=\|m(w)\|\langle I'(m(w)), z\rangle
\end{equation*}
for all $z\in T_w(S_1)=\{u\in H^1(\R^N), \langle u, w\rangle=0\}$.

$(2)$ If $\{w_n\}$ is a Palais-Smale sequence for $\Psi$, then $\{m(w_n)\}$ is a Palais-Smale sequence for $I$. If $\{v_n\}\subset\N$ is a bounded Palais-Smale sequence for $I$, then $\{m^{-1}(v_n)\}$ is a Palais-Smale sequence for $\Psi$.

$(3)$ $w$ is a critical point of $\Psi$ if and only if $m(w)$ is a nontrivial critical point of $I$. Moreover, the corresponding values of $\Psi$ and $I$ are coincide and $\inf\limits_{S_1}\Psi=\inf\limits_{\N}I$.

\end{lemma}

\begin{remark}
\rm{Lemma \ref{le7} is also similar to \cite{AS} (see Proposition 2.9 and Corollary 2.10 there) except that we do not require $I$ is coercive on $\N$, but we obtain the following lemma.}
\end{remark}

\begin{lemma}\label{le8}
Each Palais-Smale sequence for $I$ is bounded in $E$.
\end{lemma}

\begin{proof}
Let $\{v_n\}\subset E$ be a Palais-Smale sequence for $I$ at level $c\in\R$, i.e.,
\begin{equation*}
I(v_n)\rightarrow c\quad\mbox{and}\quad I'(v_n)\rightarrow 0.
\end{equation*}
Combining with Lemma \ref{le2} $(5)$, when n is large enough we have
\begin{eqnarray*}
c+o(1)&=&I(v_n)-\frac{1}{2}\langle I'(v_n), v_n\rangle
\\&=&\frac{1}{2}\int_{\R^N}h(x,v_n)v_ndx-\int_{\R^N}H(x,v_n)dx+ (\frac{1}{2}-\frac{1}{2^*})\int_{\R^N}|v_n|^{2^*}dx
\\&\geq&\frac{1}{N}\int_{\R^N}|v_n|^{2^*}dx.
\end{eqnarray*}
Hence by using Lemma \ref{le2} $(6)$, when n is large enough, one has
\begin{eqnarray*}
\|v_n\|^2&=&\langle I'(v_n), v_n\rangle+\int_{\R^N}h(x,v_n)v_ndx+\int_{\R^N}|v_n|^{2^*}dx
\\&\leq&\delta\int_{\R^N}|v_n|^2dx+C_\delta\int_{\R^N}|v_n|^{2^*}dx+\int_{\R^N}|v_n|^{2^*}dx
\\&\leq&C_3\delta\|v_n\|^2+C_\delta^{'} ,
\end{eqnarray*}
where $C_3$ and $C_\delta^{'}$ are positive constant. Fix $\delta=\frac{1}{2C_3}$, we obtain $\{v_n\}$ is bounded in $E$. Then the proof is completed.
\end{proof}

\section{Estimates}

In this section we will verify that the level value c is in an interval where the bounded Palais-Smale sequence cannot be vanishing. To show this result, we use appropriate test functions as the ones employed by Br\'ezis and Nirenberg in \cite{bn}.

We first recall that the best constant for Sobolev embedding $D^{1,2}(\R^N): =\{u\in L^{2^*}(\R^N): |\nabla u|\in L^2(\R^N)\}\subset L^{2^*}(\R^N)$ is given by
\begin{equation}\label{eqs}
S=\inf_{u\in D^{1,2}(\R^N)\setminus\{0\}}\frac{\int_{\R^N}|\nabla u|^2}{(\int_{\R^N}|u|^{2^*})^{2/2^*}} .
\end{equation}

Given $\varepsilon>0$, we consider the function $\omega_\varepsilon: \R^N\rightarrow\R$ defined by
\begin{equation*}
\omega_\varepsilon(x)=\frac{(N(N-2)\varepsilon)^{\frac{N-2}{4}}}{(\varepsilon+|x|^2)^{\frac{N-2}{2}}} .
\end{equation*}
We introduce a well-known fact that $\{\omega_\varepsilon\}_{\varepsilon>0}$ is a family of functions on which the infimum, that defines the best constant S, is attained. Let $\phi\in C_0^\infty(\R^N,[0,1])$ be a cut-off function satisfies, $\phi\equiv1$ for $x\in B_{\frac{\varrho}{2}}(0)$ and $\phi\equiv0$ for $x\in \R^N\setminus B_{\varrho}(0)$, where $B_{\varrho}(y):=\{x\in \R^N: |x-y|<\varrho\}\subset\Omega$. Define the test function by
\begin{equation*}
v_\varepsilon=\frac{u_\varepsilon}{\|u_\varepsilon\|_{2^*}},\quad\mbox{where}\ u_\varepsilon=\phi\omega_\varepsilon.
\end{equation*}

Then we can get the following estimations. Since their proofs are standard, we just state them here.
\begin{lemma}\label{le9}
There exists positive constant $K$ such that $v_\varepsilon(x)$ and $u_\varepsilon(x)$ satisfy the following estimations as $\varepsilon\rightarrow0$.
\begin{equation}\label{eq3.1}
\|v_\varepsilon\|_{2^*}=1,\quad\int_{\R^N}|\nabla v_\varepsilon|^2dx=S+O(\varepsilon^{\frac{N-2}{2}}),
\end{equation}
\begin{equation}\label{eq3.2}
\int_{\R^N}|v_\varepsilon|^2dx=\left\{\begin{matrix}O(\varepsilon^{\frac{1}{2}})&N=3, \\O(\varepsilon|\ln\varepsilon|)&N=4,\\O(\varepsilon)&N\geq5,\end{matrix}\right.
\end{equation}
and
\begin{equation}\label{eq3.3}
\int_{\R^N}|u_\varepsilon|^{2^*}dx=K+O(\varepsilon^{\frac{N}{2}}).
\end{equation}
\end{lemma}
\begin{lemma}\label{le10}
Suppose that $(g)$, $(V)$ and $(f_1)-(f_4)$ hold. Then there exists $v\in E\setminus\{0\}$ such that
\begin{equation*}
\max_{t>0}I(tv)<\frac{1}{N}S^{\frac{N}{2}}.
\end{equation*}
\end{lemma}

\begin{proof}
For $t>0$, we define
\begin{equation*}
N(t)=\frac{t^2}{2}\int_{\R^N}|\nabla v_\varepsilon|^2dx +\frac{t^2}{2}\int_{\R^N}V(x)v_\varepsilon^2dx -\int_{\R^N}H(x,tv_\varepsilon)dx-\frac{t^{2^*}}{2^{*}}\int_{\R^N}|v_\varepsilon|^{2^{*}}dx.
\end{equation*}
Lemma \ref{le3} implies that there exist a unique $t_\varepsilon>0$ such that $N(t_\varepsilon)=\max\limits_{t>0}N(t)$ and $N'(t_\varepsilon)=0$. We claim first that there exist positive constants $T_1$ and $T_2$ such that
\begin{equation*}
T_1\leq t_\varepsilon\leq T_2\quad\mbox{for all}\ \varepsilon>0\ \mbox{small enough}.
\end{equation*}
Indeed, if $t_\varepsilon\rightarrow0$ as $\varepsilon\rightarrow0$, one has $N(t_\varepsilon)\rightarrow0$, which is a contradiction. And if $t_\varepsilon\rightarrow+\infty$ as $\varepsilon\rightarrow0$, one has $N(t_\varepsilon)\rightarrow-\infty$, which is also a contradiction. Thus the claim holds.

By \eqref{eq3.1}, we can define for $s>0$,
\begin{equation*}
\Phi(s):=\frac{s^2}{2}\int_{\R^N}|\nabla v_\varepsilon|^2dx-\frac{s^{2^*}}{2^*}\int_{\R^N}|v_\varepsilon|^{2^*}dx=\frac{s^2}{2}\int_{\R^N}|\nabla v_\varepsilon|^2dx-\frac{s^{2^*}}{2^*}.
\end{equation*}
It is very standard to get that $\Phi(s)$ achieves its maximum at
\begin{equation*}
s_\varepsilon:=(\int_{\R^N}|\nabla v_\varepsilon|^2dx)^{\frac{1}{2^*-2}}=(\int_{\R^N}|\nabla v_\varepsilon|^2dx)^{\frac{N-2}{4}}.
\end{equation*}
Then we have
\begin{equation}\label{eq3.4}
\Phi(s_\varepsilon)=\max_{s>0}\Phi(s)=\frac{1}{N}(\int_{\R^N}|\nabla v_\varepsilon|^2dx)^{\frac{N}{2}}\leq\frac{1}{N}S^{\frac{N}{2}}+O(\varepsilon^{\frac{N-2}{2}}),
\end{equation}
in which we use \eqref{eq3.1} and the following inequality
\begin{equation*}
(b+d)^\beta\leq b^\beta+\beta(b+d)^{\beta-1}d\quad\mbox{for}\ b, d\geq0, \beta\geq1.
\end{equation*}
Now we consider
\begin{equation*}
\eta(\varepsilon)=\left\{\begin{matrix}O(\varepsilon^{\frac{1}{2}})&N=3, \\O(\varepsilon|\ln\varepsilon|)&N=4,\\O(\varepsilon)&N\geq5.\end{matrix}\right.
\end{equation*}
It view of $(V)$, \eqref{eq3.2} and \eqref{eq3.4}, we find a positive constant $C_4$ such that,
\begin{eqnarray}
\nonumber\max_{t>0}I(tv_\varepsilon)&=&\frac{t_\varepsilon^2}{2}\int_{\R^N}|\nabla v_\varepsilon|^2dx +\frac{t_\varepsilon^2}{2}\int_{\R^N}V(x)v_\varepsilon^2dx -\int_{\R^N}H(x,t_\varepsilon v_\varepsilon)dx
\\&&-\frac{t_\varepsilon^{2^*}}{2^{*}}\int_{\R^N}|v_\varepsilon|^{2^{*}}dx.
\nonumber\\\nonumber&\leq&\frac{1}{N}S^{\frac{N}{2}}+O(\varepsilon^{\frac{N-2}{2}})+ \frac{t_\varepsilon^2}{2}\|V\|_\infty\int_{\R^N}v_\varepsilon^2dx-\int_{\R^N}H(x,t_\varepsilon v_\varepsilon)dx
\\&\leq&\frac{1}{N}S^{\frac{N}{2}}+\eta(\varepsilon)\left[C_4 -\frac{1}{\eta(\varepsilon)}\int_{\R^N}H(x,t_\varepsilon v_\varepsilon)dx\right]\label{eq3.5}.
\end{eqnarray}
In order to prove Lemma \ref{le10}, we just need to verify that
\begin{equation}\label{eq3.6}
\lim_{\varepsilon\rightarrow0^+}\frac{1}{\eta(\varepsilon)}\int_{\R^N}H(x,t_\varepsilon v_\varepsilon)dx>C_4.
\end{equation}
Indeed by \eqref{eq3.3}, we have $\|u_\varepsilon\|_{2^*}\leq2K$ for $\varepsilon>0$ small enough, and then
\begin{equation}\label{eq3.7}
t_\varepsilon v_\varepsilon\geq\frac{T_1}{2K}u_\varepsilon=\frac{T_1}{2K}\omega_\varepsilon= \frac{T_1}{2K}\frac{(N(N-2)\varepsilon)^{\frac{N-2}{4}}}{(\varepsilon+|x|^2)^{\frac{N-2}{2}}}\geq Cv_\varepsilon^{-\frac{N-2}{4}}
\end{equation}
for $|x|<\varepsilon^{\frac{1}{2}}<\frac{\varrho}{2}$, where constant $C>0$. It follows from Lemma \ref{le2} $(4)$ that for any $A>0$, there exists $R(A)>0$ such that for all $(x,s)\in\Omega\times[R(A),+\infty)$,
\begin{equation*}
H(x,s)\geq\left\{\begin{matrix}As^4&N=3, \\As^2\ln s&N=4,\\As^2&N\geq5.\end{matrix}\right.
\end{equation*}
Thus by \eqref{eq3.7}, for $\varepsilon>0$ small enough, one has
\begin{equation*}
\int_{|x|<\varepsilon^{\frac{1}{2}}}H(x,t_\varepsilon v_\varepsilon)dx\geq\left\{\begin{matrix}CA\int_{|x|<\varepsilon^{\frac{1}{2}}}\varepsilon^{-1}dx= CA\varepsilon^{\frac{1}{2}}&N=3, \\CA\int_{|x|<\varepsilon^{\frac{1}{2}}}\varepsilon^{-1}\ln(C\varepsilon^{-\frac{1}{2}})dx= CA\varepsilon\ln(C\varepsilon^{-\frac{1}{2}})&N=4, \\CA\int_{|x|<\varepsilon^{\frac{1}{2}}}\varepsilon^{-\frac{N-2}{2}}dx= CA\varepsilon&N\geq5.\end{matrix}\right.
\end{equation*}
Combining with $H(x,s)\geq0$ and the arbitrariness of $A$, the above relation establish \eqref{eq3.6}. Hence we complete the proof.
\end{proof}

\begin{remark}
\rm{From Lemma \ref{le4} $(1)$ and Lemma \ref{le10}, we have $0<c=\inf\limits_\N I<\frac{1}{N}S^{\frac{N}{2}}$.}
\end{remark}

\section{Proof of Theorem \ref{th1}}

In this section, we will give the proof of Theorem \ref{th1}. Due to the possible lack of compactness for the criticality of the growth and the unboundedness of the definition domain, we will turn to the concentration-compactness lemmas to establish the existence of the positive ground state solutions.

\begin{proof}
By Ekeland's variational principle in \cite{mmt}, there exist a Palais-Smale sequence $\{w_n\}\subset S_1$ for $\Psi$ such that $\Psi(w_n)\rightarrow c$. Set $v_n=m(w_n)$, then from Lemma \ref{le7} $(2)$, $\{v_n\}\subset\N $ is a Palais-Smale sequence for $I$ and $I(v_n)\rightarrow c$. According to Lemma \ref{le8}, $\{v_n\}$ is bounded. Thus by the concentration-compactness principle in \cite{Lion1,Lion2}, we get $\{v_n\}$ is either

$(i)$ (Vanishing):
\begin{equation*}
\lim_{n\rightarrow\infty}\sup_{y\in\R^N}\int_{B_r(y)}|v_n|^2dx=0\quad\mbox{for all}\ r>0,
\end{equation*}

or $(ii)$ (Nonvanishing): There exist $r,\xi>0$ and a sequence $\{y_n\}\subset\R^N$ such that
\begin{equation*}
\lim_{n\rightarrow\infty}\int_{B_r(y_n)}|v_n|^2dx\geq\xi.
\end{equation*}
Now we divide our proof into the following two steps:

\textbf{Step 1:} We claim that $\{v_n\}$ cannot be vanishing if $c\in(0,\frac{1}{N}S^{\frac{N}{2}})$.

Suppose by contradiction that $\{v_n\}$ is vanishing, then it follows from P. L. Lions' Lemma 1.21 in \cite{mmt} that $v_n\rightarrow0$ in $L^\alpha(\R^N)$ whenever $2<\alpha<2^*$. From Lemma \ref{le2} $(6)$, we have that
\begin{equation*}
\left|\int_{\R^N}H(x,v_n)dx\right|\leq\frac{\delta}{2}\int_{\R^N}|v_n|^2dx +\frac{\delta}{2^*}\int_{\R^N}|v_n|^{2^*}dx+\frac{C_\delta}{\alpha}\int_{\R^N}|v_n|^\alpha dx,
\end{equation*}
for any $\delta>0$, $C_\delta>0$ and $\alpha\in(2,2^*)$, which gives that
\begin{equation*}
\int_{\R^N}H(x,v_n)dx=o(1)\quad\mbox{as}\ n\rightarrow\infty.
\end{equation*}
Similarly, we obtain
\begin{equation*}
\int_{\R^N}h(x,v_n)v_ndx=o(1)\quad\mbox{as}\ n\rightarrow\infty.
\end{equation*}
Since $\{v_n\}$ is a Palais-Smale sequence for the functional $I$, it follows that
\begin{eqnarray*}
c+o(1)&=&I(v_n)-\frac{1}{2}\langle I'(v_n), v_n\rangle
\\&=&-\int_{\R^N}H(x,v_n)dx-\frac{1}{2^{*}}\int_{\R^N}|v_n|^{2^{*}}dx+\frac{1}{2}\int_{\R^N}h(x,v_n)v_ndx
\\&&+\frac{1}{2}\int_{\R^N}|v_n|^{2^*}dx
\\&=&(\frac{1}{2}-\frac{1}{2^{*}})\int_{\R^N}|v_n|^{2^*}dx.
\end{eqnarray*}
Hence
\begin{equation}\label{eq4.1}
\lim_{n\rightarrow\infty}\int_{\R^N}|v_n|^{2^*}dx=Nc>0.
\end{equation}
Consider the fact that $\langle I'(v_n), v_n\rangle\rightarrow0$, hence we have
\begin{eqnarray*}
o(1)&=&\langle I'(v_n), v_n\rangle
\\&=&\int_{\R^N}\left(|\nabla v_n|^2+V(x)v_n^2\right)dx-\int_{\R^N}h(x,v_n)v_ndx-\int_{\R^N}|v_n|^{2^{*}}dx
\\&=&\|v_n\|^2-\int_{\R^N}|v_n|^{2^{*}}dx,
\end{eqnarray*}
therefore
\begin{equation}\label{eq4.2}
\lim_{n\rightarrow\infty}\|v_n\|^2=Nc.
\end{equation}
By the definition of the best constant $S$ in \eqref{eqs}, we obtain
\begin{equation*}
\int_{\R^N}|v_n|^{2^{*}}dx\leq\left(\frac{1}{S}\int_{\R^N}|\nabla v_n|^2dx\right)^{\frac{2^*}{2}}\leq\left(\frac{\|v_n\|^2}{S}\right)^{\frac{2^*}{2}}.
\end{equation*}
Passing to the limit in the above equality, then in view of \eqref{eq4.1} and \eqref{eq4.2} we can obtain
\begin{equation*}
Nc\leq\left(\frac{Nc}{S}\right)^{\frac{2^*}{2}},
\end{equation*}
that is
\begin{equation*}
c\geq\frac{1}{N}S^{\frac{N}{2}},
\end{equation*}
which is contradicts to the fact that $c<\frac{1}{N}S^{\frac{N}{2}}$. So that vanishing cannot occur.

\textbf{Step 2:}
If $\{v_n\}$ can be nonvanishing, here we take $y_n$ in $(ii)$ up to a subsequence, then there exist a larger $r>0$ and $\{y_n\}\subset\Z^N$ such that
\begin{equation*}
\int_{B_r(0)}|v_n(x+y_n)|^2dx=\int_{B_r(y_n)}|v_n|^2dx.
\end{equation*}
Define $\tilde{v}_n(x):= v_n(x+y_n)$. Then there exist a nonnegative function $v\in E$ such that up to a subsequence, $\tilde{v}_n\rightharpoonup v$ and $\lim\limits_{n\rightarrow\infty}\int_{B_r(0)}|v|^2dx\geq\xi$. Since $I$ is invariant and $\nabla I$ is equivariant with respect to the $\Z^N$-action, we have $I'(v)=0$. Obviously, $v$ is a nontrivial critical point of $I$, therefore $v\in\N$ and $I(v)\geq c$. Furthermore, by \eqref{eq9}, Lemma \ref{le2} $(5)$ and the Fatou's lemma, we have
\begin{eqnarray*}
c+o(1)&=&I(v_n)-\frac{1}{2}\langle I'(v_n), v_n\rangle
\\&=&\left(\frac{1}{2}\int_{\R^N}h(x,v_n)v_ndx-\int_{\R^N}H(x,v_n)dx\right)
+\frac{1}{N}\int_{\R^N}|v_n|^{2^{*}}dx
\\&\geq&\left(\frac{1}{2}\int_{\R^N}h(x,v)vdx-\int_{\R^N}H(x,v)dx\right)
+\frac{1}{N}\int_{\R^N}|v|^{2^{*}}dx
\\&=&I(v)-\frac{1}{2}\langle I'(v), v\rangle+\ o(1)
\\&=&I(v)+\ o(1),
\end{eqnarray*}
which implies $I(v)\leq c$. Hence $I(v)=c$ and thus $v$ is a positive ground state solution of problem \eqref{eq1}.

Summing the above two steps, we complete the proof of Theorem \ref{th1}.
\end{proof}

\begin{flushleft}
\textbf{Conflict of interest} The authors declared that they have no conflict of interest.
\end{flushleft}

%%%%%%%%%%%%%%%%%%%%%%%%%%%%%%%%%%%%%%%%%%¸½Â¼%%%%%%%%%%%%%%%%%%%%%%%%%%%%%%%%%%%%%%%%%%%%%%%%%
%\appendix
%
% \section{Some Examples 1}
%  some text...
% \section{Some Examples 2}
%  some text..

\bibliographystyle{elsarticle-num}
\bibliography{<your-bib-database>}

\end{document}